\documentclass[11pt]{article}
\usepackage{amsmath,amsfonts,amssymb,amsthm,enumerate,graphicx}
\usepackage{tikz,authblk}

\usepackage{pgfbasepatterns} 

\usepackage{subfig}
\usepackage{url}

\newtheorem{theorem}{{\bf Theorem}}

\newtheorem{corollary}[theorem]{{\bf Corollary}}

\newtheorem{example}[theorem]{{\bf Example}}
\newtheorem{lemma}[theorem]{{\bf Lemma}}

\newtheorem{proposition}[theorem]{{\bf Proposition}}
\newtheorem{remark}[theorem]{{\bf Remark}}

\newtheorem{question}{{\bf Question}}

\begin{document}



\title{Constructions of $d$-spheres from $(d-1)$-spheres and $d$-balls with same set of vertices}

\author{Basudeb Datta}

\affil{Department of Mathematics, Indian Institute of Science, Bangalore 560\,012, India.
 dattab@iisc.ac.in.}

\date{June 30, 2020}

\maketitle

\vspace{-10mm}

\begin{abstract}
Given a combinatorial $(d-1)$-sphere $S$, to construct a combinatorial $d$-sphere $S^{\hspace{.2mm}\prime}$ containing $S$, one usually needs some more vertices. Here we consider the question whether we can do one such construction without the help of any additional vertices. We show that this question has  affirmative answer when $S$ is  a flag sphere, a stacked sphere or a join of spheres. We also consider the question whether we can construct an $n$-vertex combinatorial $d$-sphere containing a given $n$-vertex combinatorial $d$-ball. 
\end{abstract}

\noindent {\small {\em MSC 2020\,:} 57Q15, 57C15, 57D05.

}

\section{Introduction}

If $S$ is an $n$-vertex combinatorial $(d-1)$-sphere then the join $S^0_2\ast S$ is an $(n+2)$-vertex combinatorial $d$-sphere and $S\subset S^0_2\ast S$. We also know that there is an $(n+1)$-vertex combinatorial $d$-sphere $S^{\hspace{.2mm}\prime}$ (namely, an one-point suspension of $S$) such that $S\subset S^{\hspace{.1mm}\prime}$.
In the open problem session of the Oberwolfach workshop on `Geometric, Algebraic and Topological Combinatorics' (\cite[Open Problem No. 12]{MFO2019}), Francisco Santos asked the following  

\begin{question}  \label{qn:sphere}
Given a combinatorial $(d-1)$-sphere  $S$ with $n \geq d+2$ vertices, does there always exist an $n$-vertex combinatorial $d$-sphere $S^{\hspace{.1mm}\prime}$  with $S \subset S^{\hspace{.1mm}\prime}$?
\end{question}

This of course implies that both $S$ and $S^{\hspace{.2mm}\prime}$ have same vertex-set.
Since the number of vertices in a $d$-sphere is at least $d+2$, $n$ must be at least $d+2$. 
Since a $0$-sphere consists of two points,  $n\geq d+2$  implies that $d-1 \geq 1$.  
Recently,  Criado and Santos observed (see \cite[Proof of Theorem 3.5]{CS2020}) that Question \ref{qn:sphere} has affirmative answer if $S$ is a polytopal sphere. They showed  

\begin{theorem}  \label{th:polytopal}
Let $S$ be an $n$-vertex polytopal $(d-1)$-sphere. If $n\geq d+2$ then $S$ is a subcomplex of an $n$-vertex polytopal  $d$-sphere.
\end{theorem} 

As a consequence of Theorem \ref{th:polytopal} we get

\begin{corollary} \label{cor:stacked}
Let $X^{\hspace{.2mm}d-1}$ be an $n$-vertex polytopal $(d-1)$-sphere. Then there exists a sequence $X^{\hspace{.2mm}d-1}  \subset \cdots \subset X^{\hspace{.2mm}n-2} = S^{\hspace{.2mm}n-2}_n$ of $n$-vertex polytopal spheres of dimensions $d-1, \dots, n-2$ respectively.
\end{corollary} 

Since Combinatorial 2-spheres are polytopal, from Theorem \ref{th:polytopal} we get 

\begin{corollary}  \label{cor:d=2}
Let $S$ be an $n$-vertex combinatorial $2$-sphere. If $n\geq 5$ then $S$ is a subcomplex of an $n$-vertex polytopal $3$-sphere.
\end{corollary} 

We show that Question \ref{qn:sphere} has affirmative answer in some more cases. We prove

\begin{theorem}  \label{th:flag}
Let $S$ be an $n$-vertex combinatorial $(d-1)$-sphere. Then $S$ is a subcomplex of an $n$-vertex combinatorial $d$-sphere in each of the following cases. 
\begin{enumerate}[{$(a)$}]
\item $S$ is a join of two or more spheres. 
\item  $S$ has a vertex of degree $d$ (i.e., a vertex of minimum degree), and $n\geq d+2$. 
\item $S$ is a flag sphere. 
\end{enumerate}
\end{theorem}

\begin{theorem}  \label{th:stacked}
Let $S$ be an $n$-vertex stacked $(d-1)$-sphere. If $n\geq d+2$ then $S$ is a subcomplex of an $n$-vertex stacked $d$-sphere.
\end{theorem}

As a consequence of Theorem \ref{th:stacked} we get

\begin{corollary} \label{cor:stacked}
Let $X^{\hspace{.2mm}d - 1}$ be an $n$-vertex stacked $(d-1)$-sphere. Then there exists a sequence $X^{\hspace{.2mm}d-1} \subset  \cdots \subset X^{\hspace{.2mm}n-2} = S^{\hspace{.2mm}n-2}_n$ of $n$-vertex stacked spheres of dimensions $d-1,  \dots, n-2$ respectively.
\end{corollary} 

To prove Theorem \ref{th:stacked}, we prove 

\begin{lemma} \label{lem:stacked} 
Let $B$ be an $n$-vertex  stacked $d$-ball. If $d\geq 2$ and $n\geq d+2$ then $B$ is a subcomplex of an $n$-vertex stacked $d$-sphere $S$. 
\end{lemma} 

This motivate us to  consider the following (stronger) question:  

\begin{question}  \label{qn:ball}
Given a combinatorial $d$-ball  $B$ with $n \geq d+2$ vertices, does there always exist an $n$-vertex combinatorial $d$-sphere $S$  with $B \subset S$?
\end{question}

Thus, Question \ref{qn:ball} has affirmative answer if $B$ is a stacked ball. 
We show that Question \ref{qn:ball} has  affirmative answer in some more cases. We prove

\begin{theorem}  \label{th:ball}
Let $B$ be an $n$-vertex combinatorial $d$-ball. If $n\geq d+2$ then $B$ is a subcomplex of an $n$-vertex combinatorial $d$-sphere in each of the following cases. 
\begin{enumerate}[{$(a)$}]
\item  $B$ has a vertex of degree $d$. 
\item $d=2$. 
\end{enumerate}
\end{theorem}

We also present a short proof of Theorem \ref{th:polytopal}. 
To prove Theorem \ref{th:flag} $(b)$ (respectively,  $(c)$)  we  construct  an $(n-1)$-vertex combinatorial $(d-1)$-sphere and for Theorem \ref{th:polytopal}, we construct  an $(n-1)$-vertex polytopal $(d-1)$-sphere.
In each case, the new $(n-1)$-vertex sphere contains  the anti-star of a vertex  in the original $n$-vertex sphere and an one-point suspension of this new sphere is a desired sphere. 

In Section \ref{example}, we present some examples to show that Question \ref{qn:sphere} has affirmative answer  for some non-polytopal spheres also.



\section{Preliminaries} 

We identify a simplicial complex with the corresponding abstract simplicial complex. Thus, a simplex (or face) in a simplicial complex $X$ is a finite subset of the vertex-set $V(X)$ of $X$.  We also denote a simplex $\{v_1, \dots, v_m\}$ by $v_1\cdots v_m$. We identify two simplicial complexes if they are isomorphic.  If $d$ is the dimension of a simplicial complex $X$, then  the number of $j$-faces of $X$ is denoted by $f_j(X)$ for $0\leq j\leq d$. The number of edges in $X$ through a vertex $v$ is called the {\em degree} of $v$ in $X$ and is denoted by $\deg_X(v)$. 

A simplicial complex $X$ is called {\em pure} if all its maximal faces are of same dimension. The faces of dimensions $d$ and $d-1$ in a pure $d$-dimensional simplicial complex are called {\em facets} and {\em ridges} respectively. The dual graph $\Lambda(X)$ of a pure simplicial complex $X$ is the graph whose vertices are the facets of $X$, where two facets are adjacent in $\Lambda(X)$ if they intersect in a ridge. A $d$-dimensional pure simplicial complex $X$ is called a {\em $d$-pseudomanifold}  if each ridge is in  at most  two facets and the dual graph $\Lambda(X)$ is connected. A ridge which is in one facet (resp., two facets) is called a {\em boundary} (resp., {\em interior}) ridge. The {\em boundary} $\partial X$ is the pure $(d-1)$-dimensional simplicial complex whose facets are the boundary ridges of $X$. If $\partial X$ is empty then $X$ is called a {\em pseudomanifold without boundary}. Since the degree of any vertex in the dual graph $\Lambda(X)$ of $d$-pseudomanifold (resp., $d$-pseudomanifold without boundary) $X$ is at most (resp., exactly) $d+1$ and $\Lambda(X)$ is connected, it follows that a $d$-pseudomanifold can not contains properly any $d$-pseudomanifold without boundary. 

A polyhedron  is called a {\em PL $d$-ball} if it is  PL homeomorphic to a $d$-simplex. A polyhedron is called a {\em PL $(d-1)$-sphere} if it is  PL homeomorphic to the boundary of a $d$-simplex. Thus the boundary of a PL $d$-ball is a PL $(d-1)$-sphere. From Corollaries 3.13 and 3.16 in \cite{RS1982} we know the following. 

\begin{proposition} \label{prop:pl_ball} \begin{enumerate}[{\rm (a)}]
\item Suppose $B$ is a PL $d$-ball in a PL $d$-sphere $S$ then the closure $\overline{S\setminus B}$ is a PL $d$-ball. 
\item If $B_1$, $B_2$ are PL $d$-balls and $B_1\cap B_2$ is a PL $(d-1)$-ball then $B_1\cup B_2$ is a PL $d$-ball. 
\end{enumerate}
\end{proposition}

For a simplicial complex $X$, $|X|$ denotes the geometric carrier of $X$ with induced PL structure.  A simplicial complex $X$ is said to be a {\em combinatorial $d$-ball} (resp., {\em  combinatorial $d$-sphere}) if $|X|$  is a PL $d$-ball (resp., PL  $d$-sphere). Clearly, a combinatorial $d$-ball is a $d$-pseudomanifold with boundary and a combinatorial $d$-sphere is a $d$-pseudomanifold without boundary. Since the boundary of a PL $d$-ball is a PL $(d-1)$-sphere, the boundary  of a combinatorial $d$-ball is a combinatorial $(d-1)$-sphere. Clearly, the number of vertices in a combinatorial $d$-ball (resp., $(d-1)$-sphere) is at least $d+1$. A combinatorial $d$-ball with $d+1$ vertices and (hence) with one facet is called a  {\em standard $d$-ball}. The standard $d$-ball with facet $\sigma$ is denoted by $\overline{\sigma}$. The boundary of a standard $(d+1)$-ball is called a  {\em standard $d$-sphere}. Standard $d$-sphere on the vertex-set $V$ is denoted by $S^d_{d+2}(V)$ (or $S^d_{d+2}$). A combinatorial 2-ball is also called a {\em combinatorial disc}. 

If $Y$ is a proper pure $d$-dimensional subcomplex of a pure $d$-dimensional simplicial complex $X$ then the pure subcomplex   of $X$ whose facets are those facets of $X$ which are not in $Y$ is denoted by $X-Y$. Clearly, $|X-Y|$ is equal to the closure of $|X|\setminus|Y|$. 

For two simplicial complexes $X$ and $Y$ with disjoint vertex-sets, the simplicial complex $X\ast Y := X\cup Y \cup \{\alpha\cup\beta \, : \, \alpha\in X, \beta\in Y\}$ is called the {\em join} of $X$ and $Y$. If $X$ is a combinatorial $c$-sphere and $Y$ is a combinatorial $d$-sphere then $X\ast Y$ is a combinatorial $(c+d+1)$-sphere.  If $X$ is a combinatorial $c$-ball and $Y$ is a combinatorial $d$-sphere or a combinatorial $d$-ball then $X\ast Y$ is a combinatorial  $(c+d+1)$-ball (cf. \cite[Proposition 2.23]{RS1982}).

The {\em star} ${\rm st}_X(x)$,  the {\em link} ${\rm lk}_X(x)$ and the {\em anti-star} ${\rm ast}_X(x)$ of a vertex $x$ in a simplicial complex $X$ are defined as follows  

\begin{align*}
{\rm st}_X(x) & := \{\tau\in X : \{v\}\cup\tau\in X\}, \hspace{3mm}
 {\rm lk}_X(x) := \{\tau\in {\rm st}_X(x) : v\not\in \tau\}, \\
{\rm ast}_X(x) &:= \{\tau\in X: x\not\in\tau\}. 
\end{align*} 

Clearly, ${\rm st}_X(x)=\{x\} \ast  {\rm lk}_X(x)$. Since the geometric carrier of  a combinatorial $d$-ball $B$  is a PL $d$-ball, it follows that  ${\rm lk}_B(x)$ is either a combinatorial $(d-1)$-ball or a combinatorial $(d-1)$-sphere for any $x\in V(B)$. Similarly, for a vertex $x$ in a combinatorial $d$-sphere $S$, ${\rm lk}_S(x)$ is  a combinatorial $(d-1)$-sphere.  
As a consequence of Proposition \ref{prop:pl_ball}, we get 

\begin{lemma} \label{lem:ball} 
\begin{enumerate}[{\rm (a)}]
\item If a combinatorial $d$-ball $B$ is a subcomplex of a combinatorial $d$-sphere $S$, then the subcomplex $S- B$ is also a combinatorial $d$-ball.  
\item If $B_1$, $B_2$ are combinatorial $d$-balls and $B_1\cap B_2$ is a combinatorial $(d-1)$-ball, then $B_1\cup B_2$ is a combinatorial $d$-ball. 
\item If $B_1$, $B_2$ are combinatorial $d$-balls and $B_1\cap B_2=\partial B_1=\partial B_2$,  then $B_1\cup B_2$ is a combinatorial $d$-sphere. 
\item  If  $u$ is a vertex of a combinatorial $d$-sphere $S$, then ${\rm ast}_S(u)$ is  a combinatorial $d$-ball.  
\item  If  $u$ is a vertex of a combinatorial $d$-sphere $S$, then $\{u\}\ast{\rm ast}_S(u)$ is  a combinatorial $(d+1)$-ball and  $\partial(\{u\}\ast{\rm ast}_S(u))=S$. 

\end{enumerate} 
\end{lemma}

\begin{proof} 
Part $(a)$ follows by Proposition \ref{prop:pl_ball} $(a)$. Part $(b)$ follows by Proposition \ref{prop:pl_ball} $(b)$. 

Take a new vertex $v$ not in $B_1$ or $B_2$. Since $\partial B_1$ is a combinatorial $(d-1)$-sphere, $\{v\}\ast \partial B_1$ is a combinatorial $d$-ball. Now, $\{v\}\ast B_1$, $\{v\}\ast B_2$ are combinatorial $(d+1)$-balls and  $(\{v\}\ast B_1) \cap (\{v\}\ast B_2) = \{v\}\ast (B_1\cap B_2)= \{v\}\ast \partial B_1$. Therefore, by part $(a)$, $(\{v\}\ast B_1) \cup (\{v\}\ast B_2)$ is a combinatorial $(d+1)$-ball. Observe that $\partial((\{v\}\ast B_1) \cup (\{v\}\ast B_2)) = B_1\cup B_2$. Thus, $B_1\cup B_2$ is a combinatorial $d$-sphere. This proves part $(c)$. 

Observe that ${\rm st}_S(u)\cup {\rm ast}_S(u) = S$ and ${\rm st}_S(u)\cap {\rm ast}_S(u) = {\rm lk}_S(u)$. 
Therefore, $S-{\rm st}_S(u) ={\rm ast}_S(u)$.  Since ${\rm st}_S(u)$ is a combinatorial $d$-ball, part $(d)$ now follows by part $(a)$. 

Since ${\rm ast}_S(u)$ is a combinatorial $d$-ball, $\{u\}\ast{\rm ast}_S(u)$ is a combinatorial $(d+1)$-ball. 
Now, $\partial(\{u\}\ast {\rm ast}_S(u)) = {\rm ast}_S(u) \cup (\{u\}\ast \partial ({\rm ast}_S(u))) = {\rm ast}_S(u) \cup (\{u\}\ast {\rm lk}_S(u)) = {\rm ast}_S(u) \cup {\rm st}_S(u) =S$.  This proves part $(e)$. 
\end{proof}

A {\em clique}  in a simplicial complex $X$  is a subset $U$ of $V(X)$ such that every pair of distinct vertices of $U$ forms an edge of $X$. A combinatorial $d$-sphere $S$ is called {\em flag} if $S\neq S^d_{d+2}$ and every clique in $S$ is a simplex of $S$ (equivalently,  all minimal non-faces are of size two). 

Since a $d$-polytope is a PL $d$-ball, the boundary of a $d$-polytope is a PL $(d-1)$-sphere. This implies that the boundary complex of a simplicial $d$-polytope is a combinatorial $(d-1)$-sphere.  A combinatorial $(d-1)$-sphere is called a {\em polytopal sphere} if it is (isomorphic to) the boundary complex of a simplicial $d$-polytope. Clearly, combinatorial 1-spheres (cycles) are polytopal. Since a combinatorial 2-sphere is uniquely determined by its 1-skeleton, it follows from Steinitz's theorem that every combinatorial 2-sphere is polytopal (cf. \cite{Gr2003}). For $d\geq 3$, there are combinatorial $d$-spheres which are not polytopal (cf. \cite{BD1998}). Recall that two simplicial polytopes are {\em combinatorially equivalent} if their boundary complexes are isomorphic. For any simplicial $d$-polytope $P = {\rm conv}(P(V))\subset \mathbb{R}^d$, we can perturb (the coordinates of) the vertices ``a little" without changing the combinatorial type (cf. \cite{Zi1995}). This gives 

\begin{lemma} \label{lem:simplicial} 
Let $P = {\rm conv}(P(V))\subset \mathbb{R}^d$  be a simplicial $d$-polytope. Then $P$ is combinatorially equivalent to a $d$-polytope $Q$, where no $d+1$ vertices of $Q$ are in a hyperplane. (Therefore, for $m\geq d+1$, every set of $m$ vertices of $Q$ span a simplicial $d$-polytope.) 
\end{lemma} 

\begin{proof}
Let the vertex-set $V(P)$ of $P$ be $\{v_1, \dots, v_d, u_{d+1}, \dots, u_{n}\}$, where ${\rm conv}(\{v_1, \dots, v_d\})$ is a face.  So, the affine hull ${\rm AH}(\{v_1, \dots, v_d\})$ is a hyperplane. Perturb $u_{d+1}$, if necessary, to a get new vertex $v_{d+1}$,  so that the combinatorial type of $P$ is unchanged and $v_{d+1}$ is not in the  hyperplane ${\rm AH}(\{v_1, \dots, v_d\})$. Assume that, for $d+1\leq i\leq m<n$, we have perturbed $u_{i}$ to  make new vertex $v_{i}$, so that no $d+1$ vertices in the set $\{v_1, \dots, v_m\}$ are in a hyperplane (and hence affine hull of any $d$ vertices in the set $\{v_1, \dots, v_m\}$ is a hyperplane). Now,  perturb $u_{m+1}$ to get a new vertex $v_{m+1}$ so that the combinatorial type of the polytope is unchanged and $v_{m+1}$ is not in any of the hyperplanes ${\rm AH}(\{v_{i_1}, \dots, v_{i_d}\})$, for $1\leq i_1< \cdots <  i_d\leq m$. Then, no $d+1$ vertices in the set $\{v_1, \dots, v_m, v_{m+1}\}$ are in a hyperplane. Inductively, we get new vertices $v_{d+1}, \dots, v_{n}$, so that $Q := {\rm conv}(\{v_1, \dots, v_m\})$ is combinatorially equivalent to $P$ and no $d+1$ vertices of $Q$ are in a hyperplane. This proves the lemma. 
\end{proof}

Let $X$ be a $d$-pseudomanifold without boundary and $u\in V(X)$. Then, for a new symbol $v\not\in V(X)$, the $(d + 1)$-pseudomanifold (without boundary) $\Sigma_{u,v}(X) := (\{u\} \ast {\rm ast}_X(u)) \cup (\{v\} \ast  X) $ is called an {\em one-point suspension} of $X$ (\cite{BD2008}). If $Y \subseteq X$ are pseudomanifolds without boundaries and $u\in V(Y)$ then ${\rm ast}_Y(u) \subseteq {\rm ast}_X(u)$ and hence $\Sigma_{u,v}(Y) \subseteq \Sigma_{u,v}(X)$ for $v\not\in V(X)$. For any pseudomanifold $X$ without boundary,  $|\Sigma_{u,v}(X)|$ is the suspension of $|X|$ (see \cite[Lemma 6]{BD1998}). Therefore, if $X$ is a combinatorial  $d$-sphere then $\Sigma_{u,v}(X)$ is a combinatorial $(d + 1)$-sphere. From \cite[Corollary 2 $(b)$]{BD1998}, we know the following. 

\begin{lemma} \label{lem:1pt} 
One-point suspension $\Sigma_{u,v}(X)$ is a polytopal sphere if and only if $X$ is a polytopal sphere. 
\end{lemma}

A simplicial complex $B$ is called a {\em stacked $d$-ball} if there exists a sequence $B_1, \dots, B_m$ of pure simplicial complexes such that $B_1$ is a standard $d$-ball, $B_m=B$ and, for $2\leq i\leq m$, $B_i = B_{i-1} \cup \overline{\sigma}_i$ and $B_{i-1} \cap \overline{\sigma}_i =\overline{\tau}_i$, where $\sigma_i$ is a facet of $B_i$ and $\tau_i$ is a ridge of $B_i$. (By definition, each $B_i$ is a stacked $d$-ball.) A simplicial complex is called a {\em stacked $d$-sphere} if it is (isomorphic to) the boundary of a stacked $(d + 1)$-ball. A trivial induction on $m$ shows, by Lemma \ref{lem:ball} $(b)$,  that a stacked ball is a combinatorial ball and hence a stacked sphere is  a combinatorial sphere.  Also inductively, by taking the new vertex in $\sigma_i$ which is not in $\tau_i$ outside $|B_{i-1}|$ and very close to $|\tau_i|$, we can see that $|B|$ is a polytope. Thus, a stacked sphere is a polytopal sphere.  If $B$ is a stacked ball then clearly $\Lambda(X)$ is a tree. From \cite[Lemma 2.1]{DS2013} we know

\begin{lemma} \label{lem:st_ball} 
Let $X$ be a pure simplicial complex of dimension $d$. \begin{enumerate}[{$(a)$}] 
\item If the dual graph $\Lambda(X)$ is a tree then $f_0(X) \leq f_d(X) + d$.
\item The graph $\Lambda(X)$ is a tree and $f_0(X) = f_d(X) + d$ if and only if $X$ is a stacked $d$-ball. 
\end{enumerate}
\end{lemma}


\section{Proofs}

\begin{proof}[Proof of Theorem \ref{th:polytopal}] 
Suppose $S=\partial P$, where $P$ is an $n$-vertex simplicial $d$-polytope in $\mathbb{R}^{d}$. 
Because of Lemma \ref{lem:simplicial}, we assume that that no $d+1$ vertices of $P$ are in a hyperplane. 
Let $v\in V(S)$. Since $n-1\geq d+1$, this implies that the convex hull $\widetilde{P} := {\rm conv}(V(P) \setminus\{v\})$ is a simplicial $d$-polytope in $\mathbb{R}^{d}$.  Now, if $\gamma$ is a face of $P$ not containing $v$ (i.e., $\gamma\in{\rm ast}_S(v)$) then $\gamma = P\cap W$ for some hyperplane $W$ and $P$ lies in one half space $W^{+}$ determined by $W$. Then $\widetilde{P}$ is also in $W^{+}$. Also, $\gamma \subseteq W \cap \widetilde{P} \subseteq W \cap P = \gamma$. So, $W \cap \widetilde{P} = \gamma$. Therefore $\gamma$ is a face of $\widetilde{P}$. These imply that ${\rm ast}_S(v)$ is a subcomplex of the polytopal $(d - 1)$-sphere $\widetilde{S} := \partial \widetilde{P}$. Now, $v \not\in V(\widetilde{S})$.  Choose a vertex $u$ of  $\widetilde{S}$. Consider the one-point suspension $S^{\hspace{.1mm}\prime} : = \Sigma_{u, v}\widetilde{S}$ of $\widetilde{S}$.  Since $\widetilde{S}$ is an $(n-1)$-vertex polytopal $d$-sphere, by Lemma \ref{lem:1pt}, $S^{\hspace{.1mm}\prime}$ is an $n$-vertex polytopal $d$-sphere. Since ${\rm lk}_S(v)\subset {\rm ast}_S(v) \subset \widetilde{S}$, ${\rm st}_S(v) = \{v\} \ast {\rm lk}_S(v)  \subset \{v\} \ast \widetilde{S} \subset S^{\hspace{.1mm}\prime}$. Also, ${\rm ast}_S(v) \subset \widetilde{S} \subset \Sigma_{u, v}\widetilde{S} = S^{\hspace{.1mm}\prime}$. Thus, $S= {\rm st}_S(v) \cup {\rm ast}_S(v) \subset S^{\hspace{.1mm}\prime}$. This completes the proof. 
\end{proof} 
 

\begin{proof}[Proof of Theorem \ref{th:flag}]  
Suppose  $S = S_1\ast \cdots\ast S_m$ is a join of $m$ combinatorial spheres, where $S_i$ is an $n_i$-vertex combinatorial $d_i$-sphere for $1\leq i\leq m$  and $m\geq 2$. Then $d -1 = \dim(S) = d_1+\cdots + d_m + m-1$ and $n=n_1+\cdots+n_m \geq (d_1+2)+\cdots+ (d_m+2) = d_1+\cdots + d_m +2m=d+m \geq d+2$. For $1\leq i\leq m$, choose $v_i\in V(S_i)$. By Lemma \ref{lem:ball} $(e)$, $D_i:= \{v_i\}\ast{\rm ast}_{S_i}(v_i)$ is a combinatorial $(d_i+1)$-ball and $\partial(D_i)=S_i$. Therefore, $B_1 := D_1\ast S_2\ast\cdots\ast S_m$,  $B_2 := S_1\ast D_2\ast S_3\ast\cdots\ast S_m$ are combinatorial $d$-balls and have common boundary $S_1\ast \cdots\ast S_m = S$. Hence, by Lemma \ref{lem:ball} $(c)$, $S^{\hspace{.1mm}\prime} := B_1\cup B_2$ is a combinatorial $d$-sphere and $S = \partial B_1 \subset B_1 \subset S^{\hspace{.1mm}\prime}$. Clearly, $V(S^{\hspace{.1mm}\prime}) = V(B_1)=V(S)$. This proves part $(a)$. 

Suppose $S$ has  a vertex $v$ of degree $d$. Then ${\rm lk}_S(v) = S^{d-2}_{d}(\sigma)$, for some $\sigma \subset V(S)$. If $\sigma$ is a face of $S$ then the standard $(d-1)$-sphere ${\rm st}_S(v)\cup\{\sigma\}$ is a proper subcomplex of $S$. This is not possible since both $S$ and ${\rm st}_S(v)\cup\{\sigma\}$ are $(d- 1)$-pseudomanifolds without boundary.  So, $\sigma$ is  not a face of $S$. Let $\widetilde{S}$ be  obtained from $S$ by removing all the faces containing $v$ and  adding $\sigma$ as a new $(d-1)$-simplex (i.e., $\widetilde{S} = (S - {\rm st}_S(v)) \cup \{\sigma\}$). (We say that $\widetilde{S}$ is obtained from $S$ by the {\em bistellar $(d-1)$-move} $v\mapsto \sigma$. This is also known as {\em collapsing} the vertex $v$.)
Then $\widetilde{S}$ is an $(n-1)$-vertex combinatorial $(d-1)$-sphere and $v\not\in V(\widetilde{S})$. Observe that $\widetilde{S} = {\rm ast}_S(v)\cup\{\sigma\}$. Let $u$ be a vertex of $\widetilde{S}$. Let $S^{\hspace{.1mm}\prime}$  be the one-point suspension $\Sigma_{u,v}(\widetilde{S})$. Then $S^{\hspace{.1mm}\prime}$ is a combinatorial $d$-sphere and $V(S^{\hspace{.1mm}\prime}) = V(S)$. Since $\sigma\in \widetilde{S}$, it follows that $\{v\}\cup\sigma\in S^{\hspace{.1mm}\prime}$ and hence $\{v\}\ast S^{d-2}_{d}(\sigma) \subset S^{\hspace{.1mm}\prime}$. Also, ${\rm ast}_S(v) \subset \widetilde{S} \subset S^{\hspace{.1mm}\prime}$. Therefore, $S = {\rm st}_S(v) \cup {\rm ast}_S(v) = (\{v\}\ast S^{d-2}_{d}(\sigma)) \cup {\rm ast}_S(v) \subset S^{\hspace{.1mm}\prime}$. This proves part $(b)$.

\begin{figure}[ht]

\tikzstyle{ver}=[]
\tikzstyle{vert}=[circle, draw, fill=black!100, inner sep=0pt, minimum width=4pt]
\tikzstyle{vertex}=[circle, draw, fill=black!00, inner sep=0pt, minimum width=4pt]
\tikzstyle{edge} = [draw,thick,-]
\centering

\begin{tikzpicture}[scale=0.24, thick]

\def\mypath{(14, 5) -- (10, 2) -- (6, 2) -- (2, 4) -- (0, 8) -- (0, 12) -- (2, 16) -- (6, 18) -- (10, 18) -- (14, 15) -- (18, 10) -- (14, 5)}

\fill   [gray, opacity=.3]                                \mypath; 

\draw (14, 5) -- (10, 2) -- (6, 2) -- (2, 4) -- (0, 8) -- (0, 12) -- (2, 16) -- (6, 18) -- (10, 18) -- (14, 15) -- (10.5, 12) -- (10, 18) -- (7, 14) -- (6, 18);

\draw  (2, 16) -- (7, 14) -- (10.5, 12) -- (10.5, 8) -- (14, 5);

\draw (10, 2) -- (10.5, 8) -- (6, 11) -- (7, 14) -- (0, 12) -- (6, 11) -- (0, 8) -- (4, 7) -- (2, 4); 

\draw (10.5, 12) -- (6, 11) -- (4, 7) -- (10.5, 8) -- (8, 6) -- (4, 7) -- (6, 2) -- (8, 6) -- (10, 2);

\draw (14, 15) -- (18, 10) -- (10.5, 12); \draw (14, 5) -- (18, 10) -- (10.5, 8);

\draw[dashed] (14, 5) -- (15, 10) -- (14, 15); 

\draw[dashed] (18, 10) -- (15, 10); 

\node[ver] () at (14, 10) {$u$};   \node[ver] () at (18.3, 9) {$v$}; 

\node[ver] () at (15, 3) {$S$}; 

\end{tikzpicture} 
\hspace{1mm}
\begin{tikzpicture}[scale=0.24, thick]

\def\mypath{(14, 5) -- (10, 2) -- (6, 2) -- (2, 4) -- (0, 8) -- (0, 12) -- (2, 16) -- (6, 18) -- (10, 18) -- (14, 15) -- (10.5, 12) -- (10.5, 8) -- (14, 5)}

\fill   [gray, opacity=.3]                                \mypath;

\draw (14, 5) -- (10, 2) -- (6, 2) -- (2, 4) -- (0, 8) -- (0, 12) -- (2, 16) -- (6, 18) -- (10, 18) -- (14, 15) -- (10.5, 12) -- (10, 18) -- (7, 14) -- (6, 18);

\draw  (2, 16) -- (7, 14) -- (10.5, 12) -- (10.5, 8) -- (14, 5);

\draw (10, 2) -- (10.5, 8) -- (6, 11) -- (7, 14) -- (0, 12) -- (6, 11) -- (0, 8) -- (4, 7) -- (2, 4); 

\draw (10.5, 12) -- (6, 11) -- (4, 7) -- (10.5, 8) -- (8, 6) -- (4, 7) -- (6, 2) -- (8, 6) -- (10, 2); 

\draw (14, 15) -- (15, 10) -- (14, 5); 

\node[ver] () at (14, 10) {$u$}; \node[ver] () at (15, 3) {$D_2$}; 

\def\mypath{(17.5, 8) -- (21, 5) -- (22, 10) -- (21, 15) -- (17.5, 12) -- (17.5, 8)}

\fill[gray, opacity=.5]                                \mypath; 


\draw (17.5, 12) -- (22, 10) -- (17.5, 8) -- (21, 5) -- (22, 10) -- (21, 15) -- (17.5, 12) -- (17.5, 8); 

 \node[ver] () at (23, 10) {$u$};  \node[ver] () at (21, 3) {$D_3$};


\def\mypath{(28, 15) -- (24.5, 12) -- (24.5, 8) -- (28, 5) -- (32, 10) -- (28, 15)}

\fill   [gray, opacity=.3]                                \mypath; 



\draw (24.5, 12)  -- (32, 10) -- (28, 15) -- (24.5, 12) -- (24.5, 8) -- (28, 5) -- (32, 10) -- (24.5, 8); 

\draw[dashed] (28, 5) -- (29, 10) -- (28, 15); \draw[dashed] (32, 10) -- (29, 10); 


\node[ver] () at (27.5, 10) {$u$};  \node[ver] () at (32.3, 9) {$v$}; 

 \node[ver] () at (28, 3) {$D_1$};

\end{tikzpicture}
\caption{{Sphere $S$ and discs $D_1$, $D_2$,  $D_3$}} \label{Flag}
\end{figure}
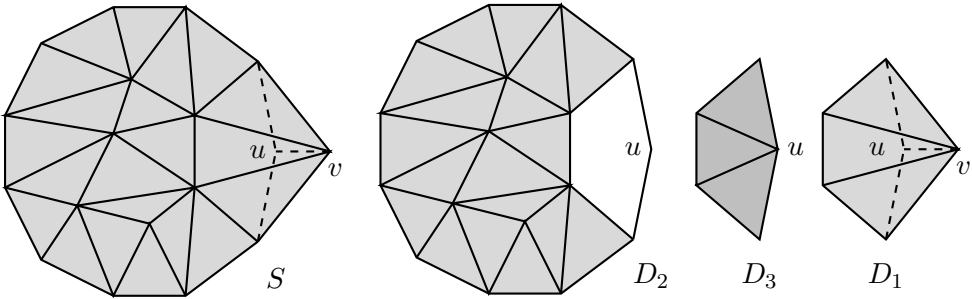

Suppose $S$ is a flag sphere of dimension $d-1$.  If $d-1=1$ then the result follows from Theorem \ref{th:polytopal} and also from part (b) above. So, assume that $d-1\geq 2$. Let $v\in V(S)$. Then $L := {\rm lk}_S(v)$ is a combinatorial sphere of dimension $d-2\geq 1$. If $L$ is a standard sphere, say $L= S^{d-2}_{d}(U)$, then $U\cup\{v\}$ is a clique and hence a face of dimension $d$, a contradiction. So, $L$ is not a standard sphere and hence $\deg_S(v) >d$. This implies that $n \geq d+2$. (In fact, if $n=d+2$ then $\deg_S(u)=d+1$ for all $u\in V(S)$ and $V(S)$ is a clique. Hence $V(S)$ is a $(d+ 1)$-simplex of $S$, a contradiction. So, $n\geq d+3$.)  Let $u\in V(L)$. Let $D_1 := {\rm st}_S(v)$, $D_2 := {\rm ast}_S(v)$ and $D_3 := \{u\}\ast {\rm ast}_L(u)$. Then, by Lemma \ref{lem:ball} $(d)$, $(e)$,  $D_1, D_2, D_3$ are combinatorial $(d-1)$-balls and $\partial D_1 = \partial D_2 = \partial D_3 = L = D_1\cap D_2 = D_1\cap D_3$. Suppose $\alpha\in D_2\cap D_3$. If $\alpha = \{w\}$ is a 0-simplex then $\alpha\in L$. If $\alpha$ is not a 0-simplex then $\alpha$ is a clique. Since the vertices of $\alpha$ are in $L= {\rm lk}_S(v)$, it follows that $\alpha\cup\{v\}$ is a clique in $S$ and hence a face of $S$. Therefore $\alpha\in L$. This implies that $D_2 \cap D_3\subseteq L$ and hence $D_2 \cap D_3 = L$. Therefore, by Lemma \ref{lem:ball} $(c)$, $S_1 := D_1\cup D_3$ and $S_2 := D_2\cup D_3$ are combinatorial $(d-1)$-spheres. (In Fig. \ref{Flag}, we present an example for $d-1=2$.)  Let $B_1 := \{v\}\ast {\rm ast}_{S_1}(v) = \{v\}\ast D_3$ and $B_2 := \{u\}\ast {\rm ast}_{S_2}(u)$. Then $B_1$, $B_2$ are combinatorial $d$-balls and $B_1\cap B_2 = D_3$. Hence $B := B_1\cup B_2$ is a combinatorial $d$-ball and $\partial B = D_1\cup D_2=S$. Let $C := \{v\}\ast {\rm ast}_{S}(v) = \{v\}\ast D_2$. Then $C$ is a combinatorial $d$-ball and $\partial C= S$. If $\beta\in C \setminus S$ then $\beta$ is of the form $\beta = \{v\}\cup \alpha$ for some $\alpha\in {\rm ast}_S(v) \setminus {\rm lk}_S(v)$. Then $\beta \not\in B_1\cup B_2 = B$. This implies that $B\cap C \subseteq S$ and hence $B\cap C = S$. Therefore, by Lemma \ref{lem:ball} $(c)$, $S^{\hspace{.1mm}\prime} := B\cup C$ is a combinatorial $d$-sphere. Clearly, $S\subset C \subset S^{\hspace{.1mm}\prime}$ and $V(S^{\hspace{.1mm}\prime}) = V(S)$. This proves part $(c)$. 
\end{proof} 

\begin{remark} \label{remark:flag} 
{\rm Let $D_2, D_3, B_1=\{v\}\ast D_3, B_2, C, S_1, S^{\hspace{.1mm}\prime}$ be as in the proof of Theorem \ref{th:flag} $(c)$. Then,  $B_1\cup C= \{v\}\ast (D_3\cup D_2) = \{v\}\ast S_2$ and $B_2 = \{u\}\ast {\rm ast}_{S_2}(u)$. Therefore, $S^{\hspace{.1mm}\prime} = 
B\cup C =   B_2\cup (B_1 \cup C)  = (\{u\}\ast {\rm ast}_{S_2}(u)) \cup (\{v\}\ast S_2) 
= \Sigma_{u,v}(S_2)$, an one-point suspension of the $(n-1)$-vertex combinatorial $(d-1)$-sphere $S_2$. This gives another (combinatorial) description of $S^{\hspace{.1mm}\prime}$. In the proof, we have given more geometric arguments, namely, we have shown that $S^{\hspace{.1mm}\prime}$ as the union of two combinatorial $d$-balls $B$,  $C$ with common boundary $S$. 
}
\end{remark}

\begin{proof}[Proof of Lemma \ref{lem:stacked}]
Since $B$ is an $n$-vertex stacked $d$-ball, there exists a sequence $B_1, \dots, B_m$ of combinatorial $d$-balls, such that $B_1=\overline{\sigma}_1$ is a standard $d$-ball, $B_m=B$ and, for $2 \leq i \leq m$,   $B_i = B_{i-1} \cup \overline{\sigma}_i$ and $B_{i-1} \cap \overline{\sigma}_i =\overline{\tau}_i$ where $\sigma_i$ is a facet in $B_i$ and $\tau_i$ is a ridge in $B_i$. By Lemma \ref{lem:st_ball}, $m = f_{d}(B) = n-d$. Then, by definition,  $C := B_{m-1}$ is a stacked $d$-ball with $m-1$ facets and $n-1$ vertices. Suppose $\sigma_m= \tau_m \cup\{u\}$. Then $V(C) = V(B)\setminus\{u\}$. Let $B^{\hspace{.1mm}\prime} = \{u\}\ast C$. Then, by the definition, $B^{\hspace{.1mm}\prime}$ is a stacked $(d+1)$-ball. (In fact,  $\Lambda(B^{\hspace{.1mm}\prime})$ is isomorphic to $\Lambda(C)$ and hence is a tree. Clearly, $B^{\hspace{.1mm}\prime}$ is a pure $(d+1)$-dimensional simplicial complex with $f_0(B^{\hspace{.1mm}\prime}) = n = m+d$ and $f_{d+1}(B^{\hspace{.1mm}\prime}) = f_{d}(C) = m-1$. Hence, by Lemma \ref{lem:st_ball}, $B^{\hspace{.1mm}\prime}$ is a stacked $(d+1)$-ball.) Then $S := \partial B^{\hspace{.1mm}\prime}$ is an $n$-vertex stacked $d$-sphere. Since $B^{\hspace{.1mm}\prime} = \{u\}\ast C$, it follows that $C \subset \partial B^{\hspace{.1mm}\prime}$. Since $\sigma_m\not\in C$, $\tau_m$ is in one facet of $C$ and hence in $\partial C$. This implies that $\sigma_m= \{u\}\cup \tau_m \in \{u\}\ast\partial C\subset \partial B^{\hspace{.1mm}\prime}$. Therefore, $B = C\cup \overline{\sigma}_m \subset \partial B^{\hspace{.1mm}\prime}=S$. This completes the proof. 
\end{proof}


\begin{proof}[Proof of Theorem \ref{th:ball}]
Let $u$ be a vertex of degree $d$ in  $B$. So, ${\rm lk}_B(u)$ has $d$ vertices and hence (since ${\rm lk}_B(u)$ is either a combinatorial $(d-1)$-ball or combinatorial $(d-1)$-sphere)  ${\rm lk}_B(u)$ is a standard $(d-1)$-ball, say $\bar{\tau}$. Then $\sigma := \{u\}\cup\tau\in B$. Since $n\geq d+2$, $B \neq \overline{\sigma}$. Then $\sigma$ has a neighbour in $\Lambda(B)$ and hence  $\tau$ is an interior ridge. Since $M := \partial B$ is a combinatorial $(d-1)$-sphere, this implies that ${\rm lk}_{M}(u)= S^{d-2}_d(\tau)$.  Let $C := \{u\}\ast {\rm ast}_{M}(u)$. Then $C$ is a combinatorial $d$-ball and $\partial C = M$. If $\beta\in C \setminus M$ then $\beta$ is of the form $\beta = \{u\} \cup \alpha$ for some $\alpha\in {\rm ast}_{M}(u) \setminus {\rm lk}_{M}(u)$. So, $\alpha \not\in S^{d-2}_d(\tau)$ and $\alpha\neq \tau$. Hence $\beta\not\in  B$. This implies that $B\cap C \subseteq M$ and hence $B\cap C = M$. Therefore, by Lemma \ref{lem:ball} $(c)$, $S := B\cup C$ is a combinatorial $d$-sphere. Clearly,  $V(S) = V(B)$. This proves part $(a)$. 

\medskip

\begin{figure}[ht]

\tikzstyle{ver}=[]
\tikzstyle{vert}=[circle, draw, fill=black!100, inner sep=0pt, minimum width=4pt]
\tikzstyle{vertex}=[circle, draw, fill=black!00, inner sep=0pt, minimum width=4pt]
\tikzstyle{edge} = [draw,thick,-]
\centering


\begin{tikzpicture}[scale=2, thick]

\def\mypath{(.707, -.707) -- (1, 0) -- (.707, .707) -- (0, 1) -- (-.707, .707) -- (-1, 0) -- (-.707, -.707) -- (0, -1)}

\fill   [gray, opacity=.25]                                \mypath; 

\draw (.707, -.707) -- (1, 0) -- (.707, .707) -- (0, 1) -- (-.707, .707) -- (-1, 0) -- (-.707, -.707) -- (0, -1);

\draw[dotted] (0, -1) -- (.707, -.707); 

\draw[dashed]  (1, 0)  -- (0, 1); \draw[dashed] (.707, .707) -- (-.707, .707) ;

\node[ver] () at (.9, -0.7) {$v_m$}; \node[ver] () at (1.2, 0) {$v_1$}; \node[ver] () at (.9, .72) {$v_2$}; 

\node[ver] () at (.2, 1.05) {$v_3$}; \node[ver] () at (-.8, .8) {$v_4$}; \node[ver] () at (-1.13, .1) {$v_5$}; 
\node[ver] () at (-.83, -.75) {$v_6$}; 


\end{tikzpicture}
\caption{{Disc with $m$ boundary vertices}} \label{Disc}
\end{figure}
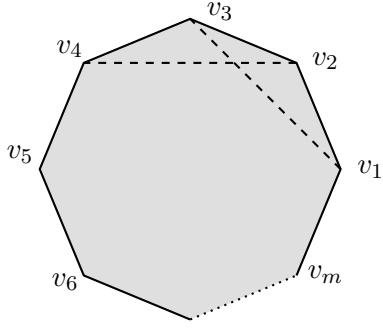

Let $B$ be a combinatorial disc with $n\geq 4$ vertices. So,  $\partial B$ is a cycle. 

\medskip

\noindent {\em Claim.} If $D_m$ is an $n$-vertex combinatorial disc whose boundary is an $m$-cycle and  $m\geq 4$, then $D_m$ is a subcomplex of an $n$-vertex combinatorial disc $D_{m-1}$ whose boundary is an $(m-1)$-cycle.  

\medskip 

Assume that $\partial D_m$ is the cycle $v_1\mbox{-}v_2\mbox{-}\cdots\mbox{-}v_m\mbox{-}v_1$. Then both $v_1v_3$ and $v_2v_4$ can not be edges of $D_m$ (see Fig. \ref{Disc}). (If $v_1v_3$ is an edge then $|D_m|\setminus |v_1v_3|$ is disconnected, one component contains $v_2$ and other component contains $v_4$ and hence $v_2v_4$ can not be an edge of $D_m$. Similarly, if $v_2v_4$ is an edge then $v_1v_3$ can not be an edge.) Assume without loss that $v_1v_3$ is a non-edge. Let $D_{m-1} := D_m\cup \overline{v_1v_2v_3} = D_m \cup \{v_1v_3, v_1v_2v_3\} $. Since $D_m\cap \overline{v_1v_2v_3}$ is a path, by Lemma \ref{lem:ball} $(b)$,  $D_{m-1}$ is a combinatorial disc. This proves the claim. 

\smallskip

By the claim, it follows that the $n$-vertex disc $B$ is a subcomplex of an $n$-vertex disc $D_3$ whose boundary is a 3-cycle. Suppose $\partial D_3 = u_1\mbox{-}u_2\mbox{-}u_3\mbox{-}u_1$. So, $u_1u_2, u_2u_3, u_1u_3$ are boundary edges. If $\alpha := u_1u_2u_3$ is a face of $D_3$ then, since each $u_1u_2, u_2u_3, u_1u_3$ is in exactly one facet,  $\alpha$ is an isolated vertex in $\Lambda(D_3)$. Since $D_3$ is a 2-pseudomanifold, it follows that $D_3 = \overline{\alpha}$. This is not possible since $n\geq 4$. So, $\alpha$ is not a face of $D_3$. Let $S := D_3 \cup \overline{\alpha}$. Then $B \subseteq D_3\subset S$ and $V(S) = V(D_3)= V(B)$. By Lemma \ref{lem:ball} $(c)$, $S$ is a combinatorial 2-sphere. These prove part $(b)$. 
\end{proof}

\begin{remark} \label{remark:delete} 
{\rm Let $B, u, \sigma, S$ be as in the proof of Theorem \ref{th:ball} $(a)$ above. Then, the induced subcomplex $B\setminus u:= B[V(B)\setminus\{u\}] = B- \overline{\sigma} = {\rm ast}_S(u)$.  Hence $B\setminus u$ is  an $(n-1)$-vertex combinatorial $d$-ball and $S= (B\setminus u)\cup (\{u\}\ast\partial(B\setminus u))$.
}
\end{remark}


\section{Some Examples}\label{example}

There are exactly two non-polytopal $3$-sphere on 8 vertices.  Here we show that Question \ref{qn:sphere} has affirmative answer  for these two non-polytopal spheres. 
Since a simplicial complex is uniquely determined by its maximal simplices, we identify a pure simplicial complex $X$ with the set of facets in $X$. 

\begin{example} \label{eg:GS} 
{\rm Consider the $8$-vertex neighbourly combinatorial $3$-sphere    
\begin{align*}
{\mathcal M}^3_8 = & \{1234, 1237, 1267, 1347, 1567, 2345,  2367,  3456, 3467, 4567, \\ 
& ~~ 1248, 1268, 1478, 1568, 1578, 2358, 2368,  2458, 3568,  4578\}. 
\end{align*}
It is a non-polytopal $3$-sphere and found by Gr\"{u}nbaum and Sreedhardan in \cite{GS1967}. Let ${\mathcal D} := {\rm ast}_{{\mathcal M}^3_8}(8)$. Then ${\mathcal D}$ is a combinatorial $3$-ball. Let ${\mathcal C} := \{1245, 1256, 1457, 2356\}$. Then ${\mathcal C}$ is a combinatorial $3$-ball and $\partial {\mathcal C} = \{157, 156, 356, 235, 245, 457, 147, 124, 126, 236\}  = {\rm lk}_{{\mathcal M}^3_8}(8) = \partial {\mathcal D}$. Therefore ${\mathcal S}^3_7 := {\mathcal C}\cup {\mathcal D}$ is a combinatorial $3$-sphere. Since ${\mathcal S}^3_7$ has 7 vertices, it is polytopal (cf.  \cite{BD1998}). Hence, by Lemma \ref{lem:1pt}, ${\mathcal S^4_8} := \Sigma_{7, 8}{\mathcal S}^3_7$ is a polytopal 4-sphere. Then, $\{8\}\ast \partial {\mathcal D} \subset \{8\}\ast {\mathcal S}^3_7 \subset {\mathcal S}^4_8$. Also, ${\mathcal D} \subset  {\mathcal S}^3_7 \subset {\mathcal S}^4_8$. Therefore, ${\mathcal M}^3_8 = {\rm st}_{{\mathcal M}^3_8}(8) \cup {\rm ast}_{{\mathcal M}^3_8}(8) = (\{8\}\ast \partial {\mathcal D}) \cup {\mathcal D} \subset {\mathcal S}^4_8$. Thus, ${\mathcal M}^3_8$ is a subcomplex of the polytopal 4-sphere ${\mathcal S}^4_8$ and $V({\mathcal S}^4_8)  = \{1,  \dots, 8\}= V({\mathcal M}^3_8)$.

For $d\geq 4$, choose $d-3$ distinct new symbols $v_9, \dots, v_{d+5}\not\in\{1, \dots, 8\}$. Let ${\mathcal M}^d_{d+5} := \Sigma_{7, v_{d+5}}(\cdots (\Sigma_{7, v_{9}}({\mathcal M}^3_8)))$ and ${\mathcal S}^{d+1}_{d+5} := \Sigma_{7, v_{d+5}}(\cdots (\Sigma_{7, v_{9}}({\mathcal S}^4_8)))$. Since ${\mathcal M}^3_{8} \subset {\mathcal S}^{4}_{8}$, inductively we get ${\mathcal M}^d_{d+5} \subset {\mathcal S}^{d+1}_{d+5}$ and  $V({\mathcal S}^{d+1}_{d+5})  = \{1,  \dots, 8, v_9, \dots, v_{d+5}\}= V({\mathcal M}^d_{d+5})$.
By Lemma \ref{lem:1pt}, ${\mathcal M}^d_{d+5}$ is a non-polytopal $d$-sphere and $V({\mathcal S}^{d+1}_{d+5})$ is a polytopal $(d+1)$-sphere. Thus, for each $d\geq 3$, there exists a  non-polytopal $d$-sphere  for which Question \ref{qn:sphere} has affirmative answer. 
}
\end{example}

\begin{example} \label{eg:Ba} 
{\rm Consider the $8$-vertex non-neighbourly combinatorial $3$-sphere    
\begin{align*}
{\mathcal M}^{\prime} = & \{2458, 2358, 1368, 1468, 2457, 2367, 1367, 1457, 1248, 1238, \\
& ~~~~ 1347, 2347, 1234, 2567, 1567, 1456, 4568, 3568, 2356 \}. 
\end{align*}
It is a non-polytopal $3$-sphere and found by Barnette  in \cite{Ba1970}. Observed that ${\mathcal M}^{\prime}$ is a subcomplex of the 8-vertex combinatorial 4-sphere ${\mathcal S}  := S^0_2(\{7, 8\}) \ast S^1_3(\{1, 2, 5\}) \ast S^1_3(\{ 3, 4, 6\})$. Thus,  Question \ref{qn:sphere} has affirmative answer for the non-polytopal 3-sphere ${\mathcal M}^{\prime} $. 

}
\end{example}

\begin{example} \label{eg:nonpolytopal} 
{\rm Let ${\mathcal M}^3_8$ and ${\mathcal D}= {\rm ast}_{{\mathcal M}^3_8}(8)$ be as in Example \ref{eg:GS}.  Then,  by Lemma \ref{lem:ball} $(b)$, $B := {\mathcal D}\cup \overline{1248}$ is a combinatorial $3$-ball with 8 vertices and $\deg_B(8)=4$. Let $S$ be the $8$-vertex combinatorial sphere containing  $B$ by our method as in the proof of Theorem \ref{th:ball} $(a)$. Since $B\setminus 8 = {\mathcal D} = {\rm ast}_{{\mathcal M}^3_8}(8)$, it follows that  $S = {\mathcal M}^3_8$  (see Remark \ref{remark:delete}). So,  our construction of a combinatorial $3$-sphere containing a combinatorial $3$-ball on the same vertex-set produces a non-polytopal $3$-sphere.}
\end{example}


\medskip

\noindent {\bf Acknowledgements:} The author thanks Mathematisches Forschung Institute Oberwolfach \& the organisers of Workshop 39/2019 for invitation and Tata Trust for travel support to attend the workshop at MFO. The author is also supported by SERB, DST (Grant No.\,MTR/2017\!/000410) and the UGC Centre for Advance Studies. The author thanks Francisco Santos for made him aware of their proof of Theorem \ref{th:polytopal} in \cite{CS2020}.


{\small

}


\begin{thebibliography}{99}
\bibitem{BD1998}
B. Bagchi and B. Datta, A structure theorem for pseudomanifolds, {\it Discrete Maths.} {\bf 188} (1998), 41--60. 

\bibitem{BD2008}
B. Bagchi and B. Datta, Lower bound theorem for normal pseudomanifolds. {\it Expositiones Math.} {\bf 26} (2008), 327--351. 

\bibitem{Ba1970}
D. Barnette, Diagrams and Schlegel diagrams, {\it Combinatorial Structures and Their Applications}, pp. 1--4, Gordon \& Breach, New York, 1970. 

\bibitem{CS2020}
F. Criado and F. Santos, Topological prismatoids and small simplicial spheres of large diameter, {\it Experimental Math.}, https://doi.org/10.1080/10586458.2019.1641766.  

\bibitem{DS2013}
B. Datta and N. Singh, An infinite family of tight triangulations of manifolds, {\it J. Combin. Theory} Ser. A
{\bf 120} (2013), 2148--2163. 

\bibitem{Gr2003} 
B. Gr\"{u}nbaum, {\it Convex polytopes}, 2nd edition (prepared and with a preface by V. Kaibel, V. Klee and G. M. Ziegler), Springer-Verlag, GTM {\bf 221}, New York, 2003. 

\bibitem{GS1967} 
B. Gr\"{u}nbaum and V. P. Sreedharan, An enumeration of simplicial 4-polytopes with 8 vertices, {\it J. Combin. Theory} {\bf 2} (1967), 437--465.

\bibitem{MFO2019}
 G. Kalai, I. Novik, F. Santos and V. Welker, {\it  Geometric, Algebraic and Topological Combinatorics}, Oberwolfach  Report No. 39/2019, DOI:10.4171/OWR/2019/39.

\bibitem{RS1982}
C. P. Rourke and B. J. Sanderson, {\it Introduction to Piecewise-Linear Topology}, Springer-Verlag, Berlin, 1982.

\bibitem{Zi1995} 
G. M. Ziegler, {\it Lectures on Polytopes}, Springer-Verlag, New York, 1995.

\end{thebibliography}
\end{document}